\documentclass[12pt]{article}
\usepackage{amsmath}
\usepackage[affil-it]{authblk}
\usepackage{amsfonts}
\usepackage{amsthm}
\usepackage{amssymb}
\usepackage{makecell}
\usepackage{appendix}
\usepackage{algorithmic}
\usepackage{caption}
\usepackage{graphicx}
\usepackage{algorithm2e}
\usepackage{tabu}
\usepackage{color}
\usepackage{soul}
\usepackage{url}
\definecolor{darkgreen}{rgb}{0.0, 0.63, 0.0}
\theoremstyle{definition}

\usepackage{mathtools}

\theoremstyle{lemma}
\newtheorem{theorem}{Theorem}

\newtheorem{lemma}[theorem]{Lemma}

\newcommand{\trace}{\text{Tr}}

\usepackage[margin=0.9in]{geometry}

\title{An Infinite Family of Real Quadratic Fields with Three Classes of Perfect Unary Forms}
\author{Christian Porter \footnote{Department of Electrical and Electronic Engineering, Imperial College London, United Kingdom. Corresponding author, c.porter17@imperial.ac.uk}}
\date{March 2024}

\begin{document}

\maketitle
\begin{abstract}
    In this paper, we revisit the theory of perfect unary forms over real quadratic fields. Specifically, we deduce an infinite family of real quadratic fields $\mathbb{Q}(\sqrt{d})$ when $d=2$ or $3$ mod $4$, such that there are three classes of perfect unary forms up to homothety and equivalence. This work, along with the work in \cite{unitred}, seems to suggest that the number of classes of perfect unary forms is related to the fundamental unit of $K$.
\end{abstract}
\section{Introduction}
Let $K$ be a totally real number field over $\mathbb{Q}$ of degree $n$. A quadratic form $f: K^m \to K$, defined by
\begin{align*}
    f(x_1,\dots,x_m)=\sum_{i,j=1}^m f_{i,j}x_ix_j
\end{align*}
is said to be positive definite if the matrix $\{\sigma(f_{i,j})\}_{i,j=1}^m$ is positive-definite for all embeddings $\sigma: K \to \mathbb{R}$. If $f$ is defined in a single variable, i.e. $m=1$, we call $f(x)=ax^2$ a unary form, which is said to be a positive unary form if $\sigma(a)>0$ for all embeddings $\sigma$, and we say that $a$ is a totally positive element of $K$. We will denote by $K_{>>0}$ the space of totally positive elements of $K$. From now on, when referring to unary forms we will always assume that the unary form is positive.
For a unary form $ax^2$ of $K$, we will denote by
\begin{align*}
    &\mu(a) \triangleq \min_{x \in \mathcal{O}_K \setminus \{0\}} \trace_{K/\mathbb{Q}}(ax^2),
    \\& \mathcal{M}(a)=\{x \in \mathcal{O}_K: \trace_{K/\mathbb{Q}}(ax^2)=\mu(a)\},
\end{align*}
where $\mathcal{O}_K$ is the ring of integers of $K$ (i.e. the minimal nonzero value and minimal vectors of the corresponding rational quadratic form generated by the algebraic trace of the unary form, respectively). A unary form $ax^2$ is said to be a perfect unary form if it is uniquely determined by $\mu(a), \mathcal{M}(a)$. Note that for any $\lambda \in \mathbb{Q}^+$, $\mu(\lambda a)=\lambda \mu(a), \mathcal{M}(\lambda a)=\mathcal{M}(a)$ and so perfect unary forms can be considered by their homothety classes. 

Let $\mathcal{O}_K^\times$ denote the unit group of $K$. We say that two unary forms $ax^2, a^\prime x^2$ are equivalent if $a^\prime=au^2$ for some $u \in \mathcal{O}_K^\times$. We denote by $n_K$ the number of $GL_1(\mathcal{O}_K)$-inequivalent homothety classes of perfect forms of a totally real number field $K$.

In \cite{Y}, if we let $K=\mathbb{Q}(\sqrt{d})$ for $d$ some square-free positive integer, the author computed $n_K$ for $d<200000$. Moreover, in the same paper it was proven that there were infinitely many $K$ such that $n_K=1$. In \cite{unitred}, it was proven that $n_K=1$ if and only if $d$ was of one of four types:

\begin{itemize}
    \item $T_1$: $d=n^2+1$, $n \in \mathbb{N}$, $n$ odd,
    \item $T_2$: $d=n^2-1$, $n \in \mathbb{N}$, $n$ even,
    \item $T_3$: $d=n^2+4$, $n \in \mathbb{N}$, $n$ odd,
    \item $T_4$: $d=n^2-4$, $n \in \mathbb{N}$, $n>3$ odd.
\end{itemize}
In the same paper, it was shown that there exists an infinite family of real quadratic fields with $n_K=2$ (namely, when $K$ is of Richaud-Degert type \cite{R-D}). In this paper, we deduce an infinite family of real quadratic fields with $n_K=3$.

\begin{theorem}\label{maintheorem}
    Let $K=\mathbb{Q}(\sqrt{d})$ for some positive, square-free integer $d$, with $d \equiv 2$ or $3 \mod 4$. Write $d=n^2+r$ for integers $n,r$ satisfying $-n < r \neq \pm 1 \leq n$. Let $\alpha+\beta \sqrt{d}$, and suppose $\beta=m(m+2)$ for some odd integer $m \geq 3$. Then:
    \begin{itemize}
        \item If $\alpha \equiv \pm 1 \mod \beta^2$, $n_K=2$. All perfect forms are homothetically equivalent to, or equivalent via some $GL_1(\mathcal{O}_K)$ transform to, either $a_1x^2$ or $a_2x^2$, where
        \begin{align*}
            &a_1=\frac{1}{2}+\frac{2n^2+r-1}{4(n^2+r)n}\sqrt{d},
            \\&a_1=\frac{1}{2}-\frac{2n^2+r-1}{4(n^2+r)n}\sqrt{d}.
        \end{align*}
        \item If $\alpha \equiv \pm \left(\frac{m-1}{2}(m+2)^2+1\right) \mod \beta^2$, $n_K=3$. All perfect forms are homothetically equivalent to, or equivalent via some $GL_1(\mathcal{O}_K)$ transform to, either $a_1x^2, a_2x^2$ or $a_3x^2$, where $a_1,a_2$ are defined as before, and
        \begin{align*}
            a_3=\alpha+\beta \sqrt{d}.
        \end{align*}
    \end{itemize}
\end{theorem}
The first point in this theorem was proven in \cite{unitred}, see Theorem 1.3 (this follows from the fact that such fields are of Richaud-Degert type). The case where $r= \pm 1$ was also tackled in the same paper. Therefore, we will only prove the second point of this theorem in this paper.
\section{Proof of Theorem \ref{maintheorem}}
First, we will prove the following useful lemma.
\begin{lemma}\label{lemma}
    Let $K=\mathbb{Q}(\sqrt{d})$ with $d>0, d \equiv 2$ or $3$ mod $4$. Suppose that $d=n^2+r$ for some integers $n,r$, $-n < r \neq \pm 1 \leq n$. Then the unary forms $a_1x^2, a_2x^2$ are perfect, with
    \begin{align*}
        &a_1=\frac{1}{2}+\frac{2n^2+r-1}{4n(n^2+r)}\sqrt{d},
        \\&a_2=\frac{1}{2}-\frac{2n^2+r-1}{4n(n^2+r)}\sqrt{d}
    \end{align*}
\end{lemma}
\begin{proof}
    It suffices to show that $a_1x^2$ has at least two minimal vectors (since $a_2$ is a conjugate of $a_1$, if $a_1x^2$ is perfect it immediately follows that $a_2x^2$ is also perfect). Note that $\trace_{K/\mathbb{Q}}(a_1)=\trace_{K/\mathbb{Q}}(a_1(n-\sqrt{d})^2)=1$. We now want to show that $\mu(a)=1$.

    Let $\gamma=\frac{2n^2+r-1}{2n(n^2+r)}$. Then
    \begin{align}\label{1}
        \trace_{K/\mathbb{Q}}(a_1(x_0+x_1\sqrt{d})^2)=(x_0+x_1d\gamma)^2+(d-d^2\gamma^2)x_1^2,
    \end{align}
    for any $x_0,x_1 \in \mathbb{Q}$. We have
    \begin{align*}
        d-d^2\gamma^2=n^2+r-\frac{(2n^2+r-1)^2}{4n^2}=1-\frac{(r-1)^2}{4n^2}>1/4,
    \end{align*}
    for all $n>0$, $-n<r \leq n$, so if $|x_1| \geq 2$, $x_0+x_1\sqrt{d} \not\in \mathcal{M}(a_1)$. Obviously if $x_0+x_1\sqrt{d} \in \mathcal{M}(a_1)$ and $x_1=0$, then $x_0=1$. Moreover, we must have $x_1 \leq 0$ since $\gamma>0$, so we assume that $x_1=-1$. 

    The minimum of $\trace_{K/\mathbb{Q}}(a_1(x_0-\sqrt{d})^2)$ is achieved for $|x_0-d\gamma| \leq 1/2$, by \eqref{1}. Now
    \begin{align}
        |x_0-d\gamma|=\left|\frac{2nx_0-(2n^2+r-1)}{2n}\right| \leq 1/2, \label{2}
    \end{align}
    when $x_0=n$, so $\{ \pm 1, \pm (n-\sqrt{d})\} \subseteq \mathcal{M}(a_1)$, so $a_1x^2$ has at least two minimal vectors, making it perfect.
\end{proof}
We will now prove Theorem \ref{maintheorem}. Assume that $K=\mathbb{Q}(\sqrt{d})$, with $d \equiv 2$ or $3 \mod 4$, $d>0$ and $d=n^2+r$, $-n<r \neq 1 \leq n$. Let $\alpha+\beta\sqrt{d}$ denote the fundamental unit of $K$, and $\beta=m(m+2)$ for some odd integer $m$, and $\alpha \equiv \pm \left(\frac{m-1}{2}(m+2)^2+1\right) \mod \beta^2$. By Lemma \ref{lemma}, the forms $a_1x^2, a_2x^2$ are perfect unary forms, where
\begin{align*}
    &a_1=\frac{1}{2}+\frac{2n^2+r-1}{4n(n^2+r)}\sqrt{d},
        \\&a_2=\frac{1}{2}-\frac{2n^2+r-1}{4n(n^2+r)}\sqrt{d}.
\end{align*}
We have $\mathcal{M}(a_1)=\{\pm 1, \pm (n-\sqrt{d})\}$ and $\mathcal{M}(a_2)=\{ \pm 1, \pm (n+\sqrt{d})\}$ if $r>-(n-1)$, or $\mathcal{M}(a_1)=\{\pm 1, \pm (n-\sqrt{d}), \pm (n-1-\sqrt{d})\}$, $\mathcal{M}(a_2)=\{ \pm 1, \pm (n-\sqrt{d}), \pm (n-1+\sqrt{d})\}$ if $r=-(n-1)$.

We say that two perfect forms $\alpha_1x^2$, $\alpha_2x^2$ are neighbours if $\mathcal{M}(a_1) \cap \mathcal{M}(a_2) \neq \phi$. By the theory of perfect forms, the graph of neighbouring forms is connected. The perfect forms $a_1x^2$, $a_2x^2$ are neighbours, so we want to look for a neighbouring form to $a_1x^2$. 

We now consider the form $a_3x^2$, where $a_3=\alpha+\beta\sqrt{d}$. We note that this form is positive, which can be seen as follows: if $\alpha+\beta\sqrt{d}$ were not totally positive, then
\begin{align*}
    \alpha^2-d\beta^2=-1 \iff \alpha^2 \equiv -1 \mod \beta^2.
\end{align*}
However, since $\beta =m(m+2)$ for some odd integer $m$, we have $\beta \equiv 3 \mod 4$, so at least one prime dividing $m$ must be equivalent to $3$ mod $4$. The equation $x^2 \equiv -1 \mod N$ has no solutions for $x \in \mathbb{Z}/N\mathbb{Z}$ if any prime $p$ equivalent to $3$ mod $4$ divides $N$, so we must instead have $\alpha^2-d\beta^2=1$.

We consider the positive rational form
\begin{align*}
    \frac{1}{2}\trace_{K/\mathbb{Q}}(a_3(x_0+x_1\sqrt{d})^2)=q(x_0,x_1) \triangleq \alpha x_0^2+ d \alpha x_1^2+2d\beta x_0x_1.
\end{align*}
We perform Gauss reduction on $q(x_0,x_1)$ in order to ascertain the minimal vectors of $q$. Consider first $\alpha=k_1m^2(m+2)^2+\frac{m-1}{2}(m+2)^2+1$, for some integer $k_1 \geq 0$. Setting $l=k_1m(m+2)+\frac{m+1}{2}$, we note that
\begin{align*}
    \left|d\frac{\beta}{\alpha}-l\right|=\left|-\frac{1}{m}+\frac{1-\alpha^{-1}}{m(m+2)}\right|<\frac{1}{m}<1/2,
\end{align*}
so $l=\lfloor d\frac{\beta}{\alpha}\rceil$. Then
\begin{align*}
    q^\prime(x_0,x_1)\triangleq q(x_0-lx_1,x_1)=\alpha x_0^2+\left( d \alpha +  l^2\alpha -2ld \beta \right)x_1^2+2\left(d\beta-l\alpha \right)x_0x_1.
\end{align*}
We now consider the quantity
\begin{align*}
    \gamma \triangleq \frac{d\beta-l\alpha}{d\alpha+l^2\alpha -2ld\beta}.
\end{align*}
First, note that $d=l^2-2k_1(m+1)-1$, and so
\begin{align*}
    &d\alpha+l^2\alpha-2ld\beta\\
    &=(2l^2-2k_1(m+1)-1)\left(k_1m^2(m+2)^2+\frac{m-1}{2}(m+2)^2+1\right)\\&-2l\left(l^2-2k_1(m+1)-1\right)m(m+2)
    \\&=k_1((m+1)^2+1)+\frac{m+1}{2},
\end{align*}
which is clearly less than $\alpha$, so $\pm 1 \not\in \mathcal{M}(a_3)$. Set $l^\prime=m+1$. Then
\begin{align}
    \gamma+l^\prime=\frac{2k_1(m+1)+1}{k_1((m+1)^2+1)+\frac{m+1}{2}} \leq \frac{1}{2}, \label{3}
\end{align}
for all $m \geq 3$, $k_1 \geq 0$, so $l^\prime=-\lfloor \gamma \rceil$. So we consider
\begin{align*}
    q^{\prime \prime}(x_0,x_1) &\triangleq q^\prime(x_0,x_1-l^\prime x_0)=\left(\alpha+{l^\prime}^2\epsilon +2l^\prime\left(d\beta-l\alpha\right)\right)x_0^2+\epsilon x_1^2+2\epsilon(\gamma+l^\prime)x_0x_1,
\end{align*}
where
\begin{align*}
    \epsilon \triangleq d\alpha+l^2\alpha-2ld\beta.
\end{align*}
A simple computation yields that
\begin{align*}
    \epsilon=\alpha+{l^\prime}^2\epsilon +2l^\prime\left(d\beta-l\alpha\right),
\end{align*}
and so we must have that $2\epsilon=\mu(a_3)$, since the Gauss reduction algorithm must conclude here. Therefore, $k_1m(m+2)+\frac{m+1}{2}-\sqrt{d}, (\alpha-\beta)\left(k_1m(m+2)+\frac{m+1}{2}+\sqrt{d}\right) \in \mathcal{M}(a_3)$. Moreover, it is easily seen that $k_1m(m+2)+\frac{m+1}{2}-\sqrt{d} \neq \pm (\alpha-\beta)\left(k_1m(m+2)+\frac{m+1}{2}+\sqrt{d}\right)$, and so $a_3x^2$ is a perfect unary form.

Suppose now $\alpha=k_2m^2(m+2)^2-\frac{m-1}{2}(m+2)^2-1$, for some integer $k_2 \geq 1$. Again, we consider the positive rational form
\begin{align*}
    \frac{1}{2}\trace_{K/\mathbb{Q}}(a_3(x_0+x_1\sqrt{d})^2)=q(x_0,x_1) \triangleq \alpha x_0^2+d\alpha x_1^2+2d\beta x_0x_1.
\end{align*}
As before, we perform Gauss reduction on $q(x_0,x_1)$. Setting $l=k_2m(m+2)-\frac{m+1}{2}$, we note that
\begin{align*}
    \left|d\frac{\beta}{\alpha}-l\right|=\left|\frac{m+1-\alpha^{-1}}{m(m+2)}\right|<\frac{1}{m}<\frac{1}{2},
\end{align*}
so $l=\lfloor d\frac{\beta}{\alpha} \rceil$. Then we set
\begin{align*}
    q^\prime(x_0,x_1)\triangleq q(x_0-lx_1,x_1)=\alpha x_0^2+\left( d \alpha +  l^2\alpha -2ld \beta \right)x_1^2+2\left(d\beta-l\alpha \right)x_0x_1.
\end{align*}
Consider
\begin{align*}
    \gamma \triangleq \gamma \triangleq \frac{d\beta-l\alpha}{d\alpha+l^2\alpha -2ld\beta}.
\end{align*}
First, note that $d=l^2+2k_2(m+1)-1$, and so
\begin{align*}
    &d\alpha+l^2\alpha-2ld\beta
    \\&=(2l^2+2k_2(m+1)-1)\left(k_2m^2(m+2)^2-\frac{m-1}{2}(m+2)^2-1\right)\\&-2l\left(l^2+2k_2(m+1)-1\right)m(m+2)
    \\&=k_2\left((m+1)^2+1\right)-\frac{m+1}{2},
\end{align*}
which is less than $\alpha$ for all $k_2 \geq 1$, and so $\pm 1 \not\in \mathcal{M}(a_3)$. Set $l^\prime=m+1$. Then
\begin{align*}
    \gamma-l^\prime=-\frac{2k_2(m+1)+1}{k_2\left((m+1)^2+1\right)-\frac{m+1}{2}}>-1/2
\end{align*}
for all $k_2 \geq 1$, $m \geq 3$, so $l^\prime=\lfloor \gamma \rceil$. It follows again that the minimum of $q(x_0,x_1)$ is $\epsilon$, where $\epsilon=d\alpha+l^2\alpha -2ld$, and so $\mu(a_3)=2\epsilon$ and $l-\sqrt{d}, (\alpha-\beta \sqrt{d})(l+\sqrt{d}) \in \mathcal{M}(a_3)$, as before.

We now want to prove that $a_1x^2, a_2x^2, a_3x^2$ are the only perfect unary forms, up to homothety and equivalence in $K$. We say that two perfect unary forms $b_1x^2, b_2x^2$ are neighbouring forms if $b_1 \neq b_2$, $\mu(b_1)=\mu(b_2)$ and $\mathcal{M}(b_1) \cap \mathcal{M}(b_2) \neq \phi$. It is known that the graph of neighbouring forms is connected \cite{mar03}.

First, we want to show the following:

\begin{align*}
    &\mathcal{M}(a_1)=\{ \pm 1, \pm (n-\sqrt{d})\},
    \\&\mathcal{M}(a_2)=\{ \pm 1, \pm (n+\sqrt{d})\},
    \\&\mathcal{M}(a_3)=\{ \pm (n-\sqrt{d}), \pm (\alpha-\beta \sqrt{d})(n+\sqrt{d})\}.
\end{align*}
Suppose that $\mathcal{M}(a_1)$ contains a minimal vector that is not $\pm 1, \pm(n-\sqrt{d})$. This can only happen when equality holds in \eqref{2}, and equality here holds only if $r=-(n-1)$. Recall that if $\alpha=km^2(m+2)^2 + \delta\left(\frac{m-1}{2}(m+2)^2+1\right)$ where $\delta \in \{\pm 1\}$, then

\begin{align*}
    d=\left(km(m+2)+\delta \frac{m+1}{2}\right)^2-2\delta k(m+1)-1.
\end{align*}
Setting $r=-2\delta k(m+1)-1$, $n=km(m+2)+\delta\frac{m+1}{2}$, it is easily seen that $-(n-1)<r<n$, and so $\mathcal{M}(a_1)=\{\pm 1, \pm (n-\sqrt{d})\}$. It can be similarly deduced that $\mathcal{M}(a_2)=\{ \pm 1, \pm (n+\sqrt{d})\}$, since $a_2$ is conjugate to $a_1$.

If there was an element in $\mathcal{M}(a_3)$ that is not equal to $\pm (n-\sqrt{d}), \pm( \alpha-\beta \sqrt{d})(n+\sqrt{d})$, then \eqref{3} would have to hold with equality. This can happen only if $k_1=0, m=3$ - however, this would mean that $d=3=4^2-1$, which contradicts the hypothesis that $d=n^2+r$ with $r \neq \pm 1$. So, $\mathcal{M}(a_3)=\{ \pm (n-\sqrt{d}), \pm (\alpha-\beta \sqrt{d})(n+\sqrt{d})\}$.

A simple computation shows that $a_1x^2, a_2x^2, a_3x^2$ do not lie in the same homothety or equivalence class, and so they must be distinct perfect unary forms up to homothety and equivalence, so $n_K \geq 3$. Recall that if $bx^2$ is a perfect unary form, then $b^\prime x^2$ is also a perfect unary form where $b^\prime$ is a conjugate of $b$. Clearly $a_1, a_2$ are conjugates. If we denote by $a_4$ the conjugate of $a_3$, it is easily seen that $a_3=a_4(\alpha-\beta \sqrt{d})^2$ and so $a_3x^2, a_4x^2$ are equivalent perfect unary forms.

Since $\mathcal{M}(a_1) \cap \mathcal{M}(a_2)= \{ \pm 1\}$, $a_1x^2$, $a_2x^2$ are neighbouring perfect forms. Similarly, $a_1x^2, a_3 x^2$ are neighbouring perfect forms. Suppose then $a_3x^2$ has another neighbouring form that is not equivalent to any of the previously listed perfect forms, call this $a^\prime x^2$. Then $ \pm (\alpha-\sqrt{d})(n+\sqrt{d}) \subseteq \mathcal{M}(a^\prime)$. This means that $ \pm (n+\sqrt{d}) \in \mathcal{M}(a^\prime (\alpha-\beta \sqrt{d})^2)$, and $a^\prime (\alpha-\beta \sqrt{d})^2 x^2$ is also a perfect unary form - however, $\pm (n+\sqrt{d}) \in \mathcal{M}(a_2)$ and $\pm (n+\sqrt{d}) \in \mathcal{M}(a_4)$ where $a_4=\alpha-\beta \sqrt{d}$, and both $a_2x^2, a_4x^2$ are perfect unary forms, so $a^\prime x^2= \lambda a_2x^2$ or $a^\prime x^2=\lambda a_4x^2$ for some $\lambda \in \mathbb{Q}^+$, i.e. $a^\prime x^2$ would lie in the same homothety and equivalence class as either $a_2x^2$ or $a_3x^2$. We therefore conclude that $n_K=3$, with all perfect unary forms being in the same homothety and/or equivalence class as either $a_1x^2, a_2x^2, a_3x^2$.


\begin{thebibliography}{xx}
\bibitem{R-D} Bae, S.: ``Real Quadratic Function Fields of Richaud-Degert Type with Ideal Class Number One''. Proc. Am. Math. Soc., vol. 140, no. 2, pp. 403--414 (2012).

\bibitem{koe60} Koecher, M.: ``Beitr\"age zu einer Reduktionstheorie in Positivit\"atsbereichen. I''. Math. Ann., vol. 141, pp. 384--432 (1960).

\bibitem{unitred} Leibak, A., Porter, C., Ling, C.: ``Unit Reducible Fields and Perfect Unary Forms''. Journal de Théorie des Nombres de Bordeaux, vol. 35, no. 6, pp. 867--895 (2024).

\bibitem{mar03} Martinet, J.: ``Perfect lattices in Euclidean spaces''. Grundlehren der mathematischen Wissenschaften, vol. 327 (2003).

\bibitem{Y}  Yasaki, D.: ``Perfect unary forms over real quadratic fields''. Journal de Th\'eorie des Nombres de Bordeaux \textbf{25} (2013) pp. 759--775. 
\end{thebibliography}
\end{document}